\documentclass{article}
\usepackage{geometry,amssymb,amsthm,amsmath,
graphics,enumerate}
\usepackage{times}
\usepackage{amsfonts}
\usepackage{verbatim}
\usepackage{amscd}
\usepackage{url}
\newfont{\msbm}{msbm10 scaled\magstephalf}
\newtheorem{theorem}{Theorem}[section]

\newtheorem{fact}[theorem]{Fact}

\def\b1K{\mbox{\boldmath $K$}_{-1}}
\def\bK{\mbox{\boldmath $K$}}

\newcommand{\cM}{\mathcal{M}}
\newcommand{\cY}{\mathcal{Y}}

\def\cA{\mathcal{A}}

\def\Col{{\rm Col}}

\newbox\noforkbox \newdimen\forklinewidth
\forklinewidth=0.3pt \setbox0\hbox{$\textstyle\smile$}
\setbox1\hbox to \wd0{\hfil\vrule width \forklinewidth depth-2pt
 height 10pt \hfil}
\wd1=0 cm \setbox\noforkbox\hbox{\lower 2pt\box1\lower
2pt\box0\relax}
\def\unionstick{\mathop{\copy\noforkbox}\limits}
\def\nonfork_#1{\unionstick_{\textstyle #1}}
\def\nonfork_#1{\unionstick_{\textstyle #1}}
\setbox0\hbox{$\textstyle\smile$} \setbox1\hbox to \wd0{\hfil{\sl
/\/}\hfil} \setbox2\hbox to \wd0{\hfil\vrule height 10pt depth
-2pt width
               \forklinewidth\hfil}
\newbox\doesforkbox
\setbox\doesforkbox\hbox{\lower 2pt\box1 \lower
2pt\box2\lower2pt\box0\relax}

\def\sub'm{\prec_{\bK'}}
\def\grpf #1 #2{{\rm grp}_{#2}(#1)}
\def\spanf #1 #2{{\rm span}_{#2}(#1)}
\def\fldf #1 #2{{\rm fld}_{#2}(#1)}
\def\dclf #1 #2{{\rm dcl}_{#2}(#1)}
\def\rclf #1 #2{{\rm rcl}_{#2}(#1)}
\def\aclf #1 #2{{\rm acl}_{#2}(#1)}
\def\acff #1 #2{{\rm acf}_{#2}(#1)}
\def\strf #1 #2{{\rm str}_{#2}(#1)}
\def\tclf #1 #2{{\rm acf}_{#2}(#1)}

\def\hbar{{\bf h}}

\date{\today}

\newtheorem{Thm}{Theorem}[section]

\newcommand{\cL}{\mathcal{L}}
\newcommand{\SP}{\mathrm{SP}}

\begin{document}

\title{The number of models of a fixed Scott rank, for a counterexample to the analytic Vaught conjecture}

%\author{ John T. Baldwin
%%\thanks{Shelah thanks the Binational Science Foundation for partial support of this research.    }
%\\University of Illinois at Chicago\\
%\and
% \and

\author{Paul B. Larson\thanks{Supported in part by NSF Grant
  DMS-1201494 and DMS-1764320.}\\
 Miami University\\
Oxford, Ohio USA\\
\and Saharon Shelah\thanks{Research partially support by NSF grant no: DMS 1101597, and
   German-Israeli Foundation for Scientific Research \& Development
   grant no. 963-98.6/2007. }\\
   %Part of F1373 on the Shelah Archive. The work in Sections \ref{upabssec} and \ref{amnotjepsec} was done around 2013. Section \ref{downctxsec} is from 2017.}\\
%   All authors thank Rutgers University.}\\
Hebrew University of Jerusalem\\Rutgers University}

\maketitle

\begin{abstract}
We show that if $\gamma \in \omega \cup \{\aleph_{0}\}$ and $\cA$ is a counterexample to the analytic Vaught conjecture having exactly 
$\gamma$ many models of Scott rank $\omega_{1}$, then there exists a club $C \subseteq \omega_{1}$ such that $\cA$ has exactly $\gamma$ many models of Scott rank $\alpha$, for each $\alpha \in C$.
\end{abstract}

Throughout this note $\tau$ represents a countable relational vocabulary. The set of $\tau$-structures with domain $\omega$ is naturally seen as a Polish space $X_{\tau}$, where a basic open set is given by the set of structures in which $R(i_{0},\ldots,i_{n-1})$ holds, for $R$ an $n$-ary relation symbol from $\tau$ and $i_{0},\ldots,i_{n-1} \in \omega$ (see Section 11.3 of \cite{Gao}, for instance). Given a sentence $\phi \in \cL_{\aleph_{1}, \aleph_{0}}(\tau)$, the set of models of $\phi$ with domain $\omega$ is a Borel subset of $X_{\tau}$.
By a theorem of Lopez-Escobar \cite{Lopez-Escobar}, every Borel subset of $X_{\tau}$ which is closed under isomporphism is also the set of models (with domain $\omega$) of some $\cL_{\aleph_{1}, \aleph_{0}}(\tau)$ sentence.

We call the following (false) statement the \emph{analytic Vaught conjecture}: for every countable relational vocabulary $\tau$, every analytic subset of $X_{\tau}$ which is closed under isomorphism and contains uncountably many nonisomorphic structures contains a perfect set of nonisomorphic structures. Steel \cite{Steel} presents two counterexamples to this statement, one due to H. Friedman and the other to K. Kunen.

Given a $\tau$-structure $M$,
%we let $\phi_{\Sc}(M)$ denote the Scott sentence of $M$, and, for each $\alpha < \omega_{1}$
we let $\SP_{\alpha}(M)$ denote the Scott process of $M$ of length $\alpha$, as defined in \cite{Larson} (this is essentially the same as the standard definition appearing in \cite{Scott, Marker2002, Marker2016}; we assume some familiarity with \cite{Larson} in the arguments below, but expect that familiarity with the classical Scott analysis will suffice). Scott's Isomorphism Theorem \cite{Scott} (rephrased) says that if $\alpha$ is a (necessarily countable) ordinal and $M$ and $N$ are countable $\tau$-structures of Scott rank at most $\alpha$, then $M$ and $N$ are isomorphic if and only if $\SP_{\alpha+1}(M) = \SP_{\alpha+1}(N)$.

Given a set $\cA \subseteq X_{\tau}$, we let $\cA^{*}$ denote the class of (ground model, but possibly uncountable) $\tau$-structures $M$ which are isomorphic to an element of the reinterpretation of $\cA$ in any (equivalently, every, by $\Sigma^{1}_{1}$-absoluteness) outer model in which $M$ is countable. If $\cA$ is the set of $\tau$-structures on $\omega$ satisfying a sentence $\phi$ of $\cL_{\aleph_{1}, \aleph_{0}}(\tau)$, then $\cA^{*}$ as defined above is simply the class of models of $\phi$.

For an ordinal $\alpha$, we let
%$\Sc^{\cA}_{\alpha}$ denote the set of Scott sentences of models structures in $\cA$ of Scott rank $\alpha$, and
$\SP_{\alpha}(\cA)$ denote the set of the Scott processes of length $\alpha$ for structures in $\cA^{*}$.
If $\cA$ is a counterexample to the analytic Vaught conjecture, then $|\SP_{\alpha}(\cA)| \leq |\alpha|$ (for $\alpha < \omega_{1}$ this follows by an induction argument using the Perfect Set Property for analytic sets; considering of a forcing extension via $\Col(\omega, \alpha)$ completes the argument for $\alpha \geq \omega_{1}$).

We also let $\cA_{\alpha}$ denote respectively the class of structures $\cA^{*}$ of Scott rank $\alpha$.
The following well-known fact (slightly restated here) appears as Corollary 10.2 in \cite{Larson}.

\begin{fact}\label{fact1} Suppose that $\mathcal{A}$ is a counterexample to the analytic Vaught conjecture, and let $x \subseteq \omega$ be such that $\mathcal{A}$ is $\Sigma^{1}_{1}$ in $x$. Let $M$ be a member of the reinterpreted version of $\mathcal{A}$ in a forcing extension of $V$, and let $\alpha$ be an ordinal. Then $\SP_{\alpha}(M) \in L[x]$.
\end{fact}

The proof of Fact \ref{fact1} given in \cite{Larson} shows the following. Similar arguments appear in Section 1 of \cite{Hjorth96} and Chapter 32 of \cite{Miller}.

\begin{Thm}\label{sec10cor2} Suppose that $\mathcal{A}$ is a counterexample to the analytic Vaught conjecture, and let $x \subseteq \omega$ be such that $\mathcal{A}$ is $\Sigma^{1}_{1}$ in $x$. Let $Y$ be a countable elementary submodel of $H((2^{\aleph_{1}})^{+})$ with $x \in Y$, let $\delta = Y \cap \omega_{1}$ and let $P$ be the transitive collapse of $Y$. Then $\SP_{\delta+1}(\cA) = \SP_{\delta+1}(\cA)^{P}$.
\end{Thm}

\begin{proof}
  Since $\cA$ is a counterexample to the analytic Vaught conjecture, $\SP_{\alpha}(\cA)^{P}$ is a countable set in $P$, for each $\alpha < \delta$. It follows that $\SP_{\alpha}(\cA)^{P} = \SP_{\alpha}(\cA)$ for each such $\alpha$, since $P$ is correct about the $\Sigma^{1}_{1}$ statement asserting that some object satisfying the conditions for membership in $\SP_{\alpha}(\cA)$ is unequal to all the members of the countable set $\SP_{\alpha}(\cA)^{P}$. Letting $g$ be $P$-generic for $\Col(\omega, \delta)$ (the partial order of finite partial functions from $\omega$ to $\delta$, ordered by inclusion), the same argument applies to show first that $\SP_{\delta}(\cA) = \SP_{\delta}(\cA)^{P[g]}$ and then that $\SP_{\delta + 1}(\cA) = \SP_{\delta + 1}(\cA)^{P[g]}$. However, each member of $\SP_{\delta}(\cA)$ in $P[g]$ must be in $P$, since otherwise there is a $\Col(\omega, \delta)$-name for an element not in $P$, and one can find perfectly many generic filters for $P$ giving distinct realizations of this name. The same argument again shows that each member of $\SP_{\delta+1}(\cA)$ in $P[g]$ must be in $P$.
\end{proof}

Theorem \ref{sec10thrm1} below can also be proved using material from \cite{Shelah43}.

%Theorem \ref{sec10thrm1} below extracts another fact from essentially the same argument.

%Theorem \ref{sec10thrm1} says something slightly stronger than the theorem stated in the abstract, since if there are only countably many models %in $\cA^{*}$ of Scott rank $\omega_{1}$ then $\SP_{\omega_{1}}(\cA)$ is countable, but (as far as we know) the converse may not hold.

\begin{Thm}\label{sec10thrm1} Suppose that $\mathcal{A}$ is a counterexample to the analytic Vaught Conjecture and $\gamma \in \omega \cup \{\aleph_{0}\}$ is such that there are up to isomorphism exactly $\gamma$ many elements of $\cA^{*}$ of Scott rank $\omega_{1}$.
Then for club many $\alpha < \omega_{1}$ there are exactly $\gamma$ many models in $\cA$ of Scott rank $\alpha$, up to isomorphism.
\end{Thm}

\begin{proof}
  Let $\cM = \{ M_{n}: n \leq \gamma \}$ be pairwise nonisomorphic elements of $\cA_{\omega_{1}}$ such that every element of $\cA_{\omega_{1}}$ is isomorphic to some element of $\cM$.
  Let $\cY$ be the set of countable elementary substructures of $H((2^{\aleph_{1}})^{+})$ containing (as elements) $\cM$ and a (fixed) code for $\cA$. We show that for each $Y \in \cY$, letting $\cM_{Y}$ be the image of $\cM$ under the transitive collapse of $Y$, every element of
  $\cA_{Y \cap \omega_{1}}$ is isomorphic to an element of $\cM_{Y}$. As the members of $\cM_{Y}$ will be nonisomorphic, this will establish the theorem.

  %the ordinal $X \cap \omega_{1}$ satisfies the conclusion of the Theorem.
  %$|\SP_{X \cap \omega_{1}}(\cA)| = |\SP_{\omega_{1}}(\cA)|$, which by Scott's Isomorphism Theorem \cite{Scott} shows that there are exactly %$|\SP_{\omega_{1}}(\cA)|$ many models in $\cA$ of Scott rank $X \cap \omega_{1}$, up to isomorphism.

  Fix $Y \in \cY$, let $\alpha = Y \cap \omega_{1}$ and let $P$ be the transitive collapse of $Y$.
  By Theorem \ref{sec10cor2}, $\SP_{\alpha + 1}(\cA) = \SP_{\alpha +1}(\cA)^{P}$. Suppose toward a contradiction that there exists an
  $N \in \cA_{\alpha} \setminus \cM_{Y}$. Then $\SP_{\alpha + 1}(N) \in P$. Proposition 5.19 of \cite{Larson} then implies that the $\delta$-th level of $\SP_{\delta + 1}(N)$ amalgamates, as defined in Definition 5.16 of \cite{Larson}. Since amalgamation is a first order property it is witnessed in $P$. It follows from Propositon 7.10 of \cite{Larson} that there is a model of $\SP_{\delta + 1}(N)$ in $P$, contradicting the elementarity of the collapse and the assumed property of $\cM$.
\end{proof}

The proofs of Theorems \ref{sec10cor2} and \ref{sec10thrm1} can be used to prove the following variation, which we leave to the interested reader : there is fragment $T$ of ZFC such that, if $x \subseteq \omega$ is a code for an analytic class $\cA$ of $\tau$-structures then for any transitive model $P$ of $T$ containing $x$ and any ordinal $\alpha$ in $P$, if $\SP_{\alpha}(\cA)$ is countable then $\SP_{\alpha}(\cA) = \SP_{\alpha}(\cA)^{P}$, and, if in addition $\alpha < \omega_{2}^{P}$, then every structure in $\cA^{*}$ of Scott rank $\alpha$ is isomorphic to one in $P$.

\end{document}